\documentclass[11pt]{article}
\usepackage{amsmath, amsthm, amsfonts, amssymb,algorithm2e,caption,enumerate}
\usepackage{fullpage}
\usepackage{hyperref}
\usepackage{xcolor}
\usepackage{tikz}
\usepackage{enumerate}
\usetikzlibrary{decorations.pathreplacing} 
\usepackage{a4wide}

\usepackage{xspace}
\usepackage{comment}
\usepackage{cleveref}

\title{\vspace{-0.9cm} Powers of paths in tournaments}
\author{
Nemanja Dragani\'c\thanks{Department of Mathematics, ETH Zurich, Zurich, Switzerland. Emails: \texttt{\{nemanja.draganic, david.munhacanascorreia, benjamin.sudakov\}@math.ethz.ch}. Research supported in part by SNSF grant 200021\_196965}
\and
Fran\c{c}ois Dross\thanks{Universit\'e de Bordeaux, CNRS, Bordeaux INP, LaBRI, UMR 5800, F-33400, Talence, France. Email: \texttt{francois.dross@u-bordeaux.fr}. Research supported in part by ERC grant No 714704.}
\and
Jacob Fox\thanks{Department of Mathematics, Stanford University, Stanford, CA, USA. Email: \texttt{jacobfox@stanford.edu}. Research supported by a Packard Fellowship and by NSF award DMS-185563.}
\and
Ant\'onio Gir\~ao\thanks{Institut f\"ur Informatik, Universit\"at Heidelberg, Germany. E-mail: \texttt{tzgirao@gmail.com}.}
\and
Fr\'ed\'eric Havet\thanks{CNRS, Universit\'e C\^ote d'Azur, I3S, INRIA, Sophia Antipolis, France. Email: \texttt{frederic.havet@inria.fr}. Research supported in part by Agence Nationale de la Recherche under contract Digraphs ANR-19-CE48-0013-01.}
\and
D\'aniel Kor\'andi\thanks{Mathematical Institute, University of Oxford, Andrew Wiles Building, Radcliffe Observatory Quarter, Woodstock Road, Oxford, United Kingdom.   Emails: \texttt{\{korandi, scott\}@maths.ox.ac.uk}.} \thanks{Supported by SNSF Postdoc.Mobility Fellowship P400P2\_186686.}
\and
William Lochet\thanks{Department of Informatics, University of Bergen, Norway. Email: \texttt{william.lochet@uib.no}. Research supported by ERC grant No 819416.}
\and
David Munh\'a Correia\footnotemark[1]
\and
Alex Scott\footnotemark[6]
\and
Benny Sudakov\footnotemark[1]
}

\newtheorem{thm}{Theorem}
\newtheorem{lem}[thm]{Lemma}

\Crefname{thm}{Theorem}{Theorems}
\Crefname{lem}{Lemma}{Lemmas}
\Crefname{prop}{Proposition}{Propositions}
\Crefname{conj}{Conjecture}{Conjectures}
\Crefname{cor}{Corollary}{Corollaries}
\Crefname{prob}{Problem}{Problems}

\theoremstyle{definition}

\newtheorem*{claim}{Claim}

\Crefname{defn}{Definition}{Definitions}
\Crefname{claim}{Claim}{Claims}
\Crefname{obs}{Observation}{Observations}

\newcommand{\ceil}[1]{\left\lceil #1 \right \rceil}
\newcommand{\eps}{\varepsilon}
\newcommand{\subs}{\subseteq}

\date{}

\begin{document}

\maketitle

\begin{abstract}
In this short note we prove that every tournament contains the $k$-th power of a directed path of linear length. This improves upon recent results of Yuster and of Gir\~ao. 
We also give a complete solution for this problem when $k=2$, showing that there is always a square of a directed path of length $\ceil{2n/3}-1$, which is best possible.
\end{abstract}

\section{Introduction}

One of the main themes in extremal graph theory is the study of embedding long paths and cycles in graphs. Some of the classical examples include the Erd\H{o}s--Gallai theorem \cite{ErG} that every $n$-vertex graph with average degree $d$ contains a path of length $d$, and Dirac's theorem \cite{Dir} that every graph with minimum degree $n/2$ contains a Hamilton cycle. A famous generalization of this, conjectured by P\'osa and Seymour, and proved for large $n$ by Koml\'os, S\'ark\"ozy and Szemer\'edi \cite{KSS98}, asserts that if the minimum degree is at least $kn/(k+1)$, then the graph contains the $k$-th power of a Hamilton cycle.

In this note, we are interested in embedding directed graphs in a tournament. A tournament is an oriented complete graph.
The $k$-th power of the directed path $\vec{P}_{\ell} = v_0\dots v_{\ell}$ of length $\ell$ is the graph $\vec{P}^k_{\ell}$ on the same vertex set containing a directed edge $v_iv_j$ if and only if $i<j\le i+k$. The $k$-th power of a directed cycle is defined analogously.  
An old result of Bollob\'as and H\"aggkvist \cite{BH90} says that, for large $n$,  
every $n$-vertex tournament with all indegrees and outdegrees at least $(1/4+\eps)n$ contains the $k$-th power of a Hamilton cycle (the constant $1/4$ is optimal). 
However, we cannot expect to find powers of directed cycles in general, as the transitive tournament contains no cycles at all.  

What about powers of directed paths? A classical result, which appears in every graph theory book (see, e.g., \cite{West}), says that every tournament contains a directed Hamilton path. On the other hand,
Yuster \cite{Y21} recently observed that some tournaments are quite far from containing the square of a Hamilton path. In particular,
there is an $n$-vertex tournament that does not even contain the square of $\vec{P}_{2n/3}$, and more generally, for every $k\ge2$, there are tournaments with $n$ vertices and no $k$-th power of a path with more than $nk/2^{k/2}$  vertices.
In the other direction, Yuster proved that every tournament with $n$ vertices  contains the square of a path of length $n^{0.295}$. This was improved very recently by Gir\~ao \cite{G20}, who showed that for fixed $k$, every tournament on $n$ vertices contains the $k$-th power of a path of length $n^{1-o(1)}$.   Both papers noted that no sublinear upper bound is known.  Our main result shows that the maximum length is in fact linear in $n$.

\begin{thm} \label{thm:main} For $n \geq 2$, every $n$-vertex tournament contains the $k$-th power of a directed path of length $n/2^{4k+6}k$.
\end{thm}

The proof of this theorem combines K\H{o}v\'ari--S\'os--Tur\'an style arguments, used for the bipartite Tur\'an problem, and median orderings of tournaments. A median ordering is a vertex ordering that maximizes the number of forward edges. \Cref{thm:main} and Yuster's construction show that an optimal bound on the length has the form $n/2^{\Theta(k)}$. 
It would be interesting to find the exact value of the constant factor in the exponent.
Optimizing our proof can yield a lower bound of $n/2^{ck+o(k)}$ with $c\approx 3.9$, but is unlikely to give the correct bound.

We also improve the exponential constant in the upper bound from $1/2$ to $1$.

\begin{thm} \label{thm:upper}
Let $k\ge 5$ and $n\ge k(k+1)2^k$. There is an $n$-vertex tournament that does not contain the $k$-th power of a directed path of length $k(k+1)n/2^k$.
\end{thm}

\noindent
Note that this theorem also holds trivially for $k \le 4$, when $k(k+1)n/2^k>n$.

Finally, we can solve the problem completely in the special case of $k=2$. Once again, the proof uses certain properties of median orderings.

\begin{thm} \label{thm:square}
For $n\ge 1$, every $n$-vertex tournament contains the square of a directed path of length $\ell= \ceil{2n/3}-1$, but not necessarily of length $\ell+1$.
\end{thm}

Theorems \ref{thm:main}, \ref{thm:upper} and \ref{thm:square} are proved in Sections \ref{sec:lower}, \ref{sec:upper} and \ref{sec:square}, respectively.

\section{Lower bound} \label{sec:lower}

We will need the following K\H{o}v\'ari--S\'os--Tur\'an style lemma. 
\begin{lem} \label{lem:kst} Let $G$ be a directed graph with disjoint vertex subsets $A$ and $B$ with $|A|=2k+1$, $|B| \geq 2^{4k+4}k$, and every vertex in $A$ has at least $(1-\frac{1}{2k+1})|B|/2$ outneighbours in $B$. Then $A$ contains a subset $A'$ of size $k$ that has at least $(2k+1)2^{2k}$ common outneighbours in $B$.
\end{lem}
\begin{proof}
Suppose there is no such set $A'$. Then every $k$-subset of $A$ appears in the inneighbourhood of less than $(2k+1)2^{2k}$ vertices in $B$. So if $d^-(v)$ denotes the number of inneighbours a vertex $v\in B$ has in $A$, then we have
\begin{equation}\label{eq1}
\binom{2k+1}{k} \cdot (2k+1)2^{2k} = \binom{|A|}{k}\cdot (2k+1)2^{2k} > \sum_{v\in B} \binom{d^-(v)}{k}.
\end{equation} 
On the other hand, $\sum_{v\in B} d^-(v) \ge |A|(1-\frac{1}{2k+1})|B|/2 = k|B|$. By Jensen's inequality, $\sum_{v\in B} \binom{d^-(v)}{k} \ge |B| \cdot
\binom{\sum_{v\in B} d^-(v)/|B|}{k} = |B| \geq 2^{4k+4}k$. This contradicts (\ref{eq1}). 
\end{proof}

One more ingredient we need for the proof of \Cref{thm:main} is the folklore fact that every tournament on $2^{m}$ vertices contains a transitive subtournament of size $m+1$. This is easily seen by taking a vertex of outdegree at least $2^{m-1}$ as the first vertex of the subtournament, and then recursing on the outneighbourhood.

\begin{proof}[Proof of \Cref{thm:main}]
Order the vertices as $0,1,\dots,n-1$ to maximize the number of forward edges, i.e., the number of edges $ij$ such that $i<j$. As was mentioned in the introduction, we will refer to such a sequence as a \emph{median ordering} of the vertices. We denote an ``interval'' of vertices with respect to this ordering by $[i,j) = \{i,\dots,j-1\}$, where $0\le i< j\le n$.

We will embed $\vec{P}^k_{\ell}$ inductively using the following claim.
\begin{claim}
Let $t=2^{4k+4}k$ and $t \leq i  \leq n-(2k+1)t$. For every subset $A^*\subs [i-t,i)$ of size $2^{2k}$, there is an index $i+t\le j\le i+(2
k+1)t$ and a set $A'\subs A^*$ of size $k$ such that $A'$ induces a transitive tournament and its vertices have at least $2^{2k}$ common outneighbours in $[j-t,j)$.
\end{claim}
\begin{proof}
There is a subset $A\subs A^*$ of size $2k+1$ that induces a transitive tournament. Let $B=[i,i+(2k+1)t)$. Then every vertex $v\in A$ has at least $kt=\left(1-\frac{1}{2k+1}\right)|B|/2$ outneighbours in $B$. Indeed, otherwise $v$ would have more than $(k+1)t$ inneighbours in the interval $B$, so moving $v$ to the end of this interval would increase the number of forward edges in the ordering, contradicting our choice of the vertex ordering.

We can thus apply \Cref{lem:kst} to find a $k$-subset $A'\subs A$ with least $(2k+1)2^{2k}$ common outneighbours in $B$. Partition $B$ into $2k+1$ intervals of size $t$, and we can choose $j$ accordingly so that $A'$ has at least $2^{2k}$ common outneighbors in the interval $[j-t,j)$.
\end{proof}

The theorem trivially holds for $n < 2^{2k}$, so assume $n \geq 2^{2k}$. 
Let $i_0=2^{2k}$ and $A_0=[0,2^{2k})$, and apply the Claim with $i=i_0$ and $A^*=A_0$. We get a set $A'\subset A_0$ of size $k$ that induces a transitive tournament, i.e., the $k$-th power of some path $v_0 \dots v_{k-1}$. Moreover, this $A'$ has at least $2^{2k}$ common outneighbours in some interval $[j-t,j)$ with $i_0 +t\le j \le i_0+(2k+1)t$. Let us define $i_1=j$, and choose $A_1$ to be any $2^{2k}$ of the common outneighbours.

At step $s$, we apply the Claim again with $i=i_s$ and $A^*=A_s$ to find the $k$-th power of some path $v_{sk} \dots v_{(s+1)k-1}$ in $A_s$ with $2^{2k}$ common outneighbours in some $[i_{s+1}-t,i_{s+1})$ with $i_s+t \le i_{s+1} \le i_s+(2k+1)t$, and repeat this process until some step $\ell$ with $i_{\ell}>n-(2k+1)t$. Note that intervals $[i_s-t,i_s)$ and $[i_{s+1}-t,i_{s+1})$ are always disjoint.
Finally, $A_{\ell}$ must also contain a transitive tournament of size $2k+1$. Call these vertices $v_{\ell k},\ldots,v_{(\ell+2)k}$. Observe that $n-(2k+1)t<i_{\ell} \leq 2^{2k}+\ell(2k+1)t$, so $n<(\ell+2)(2k+1)t$. 

Then $v_0\dots v_{(\ell+2)k}$ is a directed path of length $(\ell+2)k \ge kn/(2k+1)t \ge n/(2^{4k+6}k)$ whose $k$-th power is contained in the tournament. In fact, we proved a bit more:~the tournament contains all edges of the form $v_av_b$ with $a<b$ and $\lfloor a/k \rfloor +1 \geq \lfloor b/k \rfloor$.
\end{proof}

\section{Upper bound} \label{sec:upper}

Let $\ell_k(n)$ denote the smallest integer $\ell$ such that there is an $n$-vertex tournament that does not contain $\vec{P}^k_{\ell}$, or in other words, the largest integer such that every $n$-vertex tournament contains the $k$-th power of a directed path on $\ell$ vertices.

To prove \Cref{thm:upper}, we first note that $\ell_k(n)$ is subadditive.

\begin{lem} \label{lem:subadd}
For any $k,n,m\ge 1$, we have $\ell_k(n+m) \le \ell_k(n) + \ell_k(m)$.
\end{lem}
\begin{proof}
Let $T_1$ and $T_2$ be extremal tournaments on $n$ and $m$ vertices, respectively, not containing the $k$-th power of any directed path of length $\ell_k(n)$ and $\ell_k(m)$. Let $T$ be the tournament on $n+m$ vertices, obtained from the disjoint union of $T_1$ and $T_2$ by adding all remaining edges directed from $T_1$ to $T_2$.
Then any $k$-th power of a path in $T$ must be the concatenation of the $k$-th power of a path in $T_1$ and the $k$-th power of a path in $T_2$, and hence it must have length at most $(\ell_k(n)-1) + (\ell_k(m)-1) + 1 < \ell_k(n)+\ell_k(m)$.
\end{proof}

Our improved upper bound is based on the following construction.
\begin{lem} \label{lem:small}
For every $k\ge 5$, we have $\ell_k(2^{k-1}) < \frac{k(k+1)}{2}$.
\end{lem}
\begin{proof}
Let $n=2^{k-1}$ and $\ell=\frac{k(k+1)}{2}$, and note that $\vec{P}^k_{\ell-1}$ has $k\ell - \ell$ edges.

Let $T$ be a random $n$-vertex tournament obtained by orienting the edges of $K_n$ independently and uniformly at random. The probability that a fixed sequence of $\ell$ vertices $v_0 \dots v_{\ell-1}$ forms a copy of $\vec{P}^k_{\ell-1}$ is $2^{-(k-1)\ell}$. There are $\binom{n}{\ell} \cdot \ell !$ such sequences, so the probability that $T$ contains the $k$-th power of a path of length $\ell-1$ is at most 
$\binom{n}{\ell} \cdot \ell ! \cdot 2^{-(k-1)\ell} < n^{\ell} \cdot 2^{-(k-1)\ell} = 1$. 
So with positive probability $T$ does not contain $\vec{P}^k_{\ell-1}$, therefore $\ell_k(2^{k-1}) \le \ell-1$.
\end{proof}

Combining \Cref{lem:subadd,lem:small} and using the monotonicity of $\ell_k(n)$, we get
\[ \ell_k(n) \le \ceil{\frac{n}{2^{k-1}}} \cdot \ell_k(2^{k-1}) \le \left(\frac{n}{2^{k-1}} + 1\right)\left( \frac{k(k+1)}{2} -1 \right) \le \frac{k(k+1)n}{2^k} \]
for $n \ge k(k+1)2^k$, establishing \Cref{thm:upper}.

\section{The square of a path} \label{sec:square}

\begin{proof}[Proof of \Cref{thm:square}]
Recall that $\ell_2(n)$ is the largest integer such that every $n$-vertex tournament contains the square of a path on $\ell$ vertices. Proving \Cref{thm:square} is therefore equivalent to showing $\ell_2(n) = \ceil{2n/3}$ for every $n\ge 1$.

It is easy to check that $\ell_2(1)=1$ and $\ell_2(2)=\ell_2(3)=2$, so $\ell_2(n) \le \ceil{2n/3}$ follows from \Cref{lem:subadd} by induction, as $\ell_2(n) \le \ell_2(n-3) + \ell_2(3) = \ell_2(n-3) + 2$ holds for every $n>3$. For the lower bound we need to take a closer look at median orderings.

\begin{claim}
Every median ordering $x_1,\dots,x_n$ of a tournament has the following properties:
\begin{enumerate}[(a)]
 \item \label{prop:edge}
 All edges of the form $x_ix_{i+1}$ are in the tournament.
 \item \label{prop:rotate}
 If $x_ix_{i-2}$ is an edge of the tournament, then ``rotating'' $x_{i-2}x_{i-1}x_i$ gives two other median orderings $x_1,\dots,x_{i-3},x_{i-1},x_{i},x_{i-2},x_{i+1},\dots,x_n$ and $x_1,\dots,x_{i-3},x_{i},x_{i-2},x_{i-1},x_{i+1},\dots,x_n$.
 \item \label{prop:triple}
 If $x_ix_{i-2}$ is an edge of the tournament, then each of $x_{i-2},x_{i-1},x_i$  is an inneighbour of $x_{i+1}$, and at most one of them is an outneighbour of $x_{i+2}$. 
\end{enumerate}
\end{claim}
\begin{proof}
Property \eqref{prop:edge} holds, as otherwise we could swap $x_i$ and $x_{i+1}$ to get an ordering with more forward edges, contradicting our assumption. Property \eqref{prop:rotate} holds because rotating $x_{i-2}x_{i-1}x_i$ has no effect on the number of forward edges.

These two properties together imply that each of $x_{i-2},x_{i-1},x_i$  is an inneighbour of $x_{i+1}$. Suppose, to the contrary of \eqref{prop:triple}, that two of them are outneighbours of $x_{i+2}$. By rotating $x_{i-2}x_{i-1}x_i$ if needed, we may assume that these are $x_{i-1}$ and $x_i$. But then we can also rotate $x_ix_{i+1}x_{i+2}$ so that $x_{i+2}$ comes right after $x_{i-1}$ in a median ordering. This contradicts \eqref{prop:edge}.
\end{proof}

Let us now say that $i$ is a \emph{bad index} in a median ordering $x_1,\dots,x_n$ if $x_ix_{i-2}$ is an edge, and at least one of $x_{i+2}x_i$ and $x_{i+2}x_{i-1}$ is also an edge.

\begin{lem} \label{lem:median}
Every tournament has a median ordering without any bad indices.
\end{lem}
\begin{proof}
Suppose this fails to hold for some tournament, and take a median ordering $x_1,\dots,x_n$ that minimizes the largest bad index $i$. As $i$ is a bad index, $x_ix_{i-2}$ is an edge, and $x_i$ or $x_{i-1}$ is an outneighbour of $x_{i+2}$. By \eqref{prop:rotate}, $x_{i-2}x_{i-1}x_i$ can be rotated so that $x_{i+2}x'_{i-2}$ is an edge in the new median ordering  $x_1,\dots,x_{i-3},x'_{i-2},x'_{i-1},x'_i,x_{i+1},\dots,x_n$. 
Then neither $x_{i+2}x'_i$ nor $x_{i+2}x'_{i-1}$ is an edge, since by \eqref{prop:triple}, only one of $x'_{i-2},x'_{i-1},x'_i$ is an outneighbour of $x_{i+2}$. Also by \eqref{prop:triple}, $x'_{i-1}x_{i+1}$ and  $x'_ix_{i+1}$ are edges, so both of $x_{i+1}$ and $x_{i+2}$ are outneighbours of $x'_{i-1}$ and $x'_i$. This means that none of $i,i+1,i+2$ is a bad index in this new ordering, and hence the largest bad index is smaller than $i$. This is a contradiction.
\end{proof}

Now we are ready to prove $\ell_2(n) \ge \ceil{2n/3}$. Take an $n$-vertex tournament with median ordering $x_1,\dots,x_n$ as in \Cref{lem:median}, and let $I= \{i_1< i_2 <\dots <i_k\}$ be the set of indices $i$ such that $x_i x_{i-2}$ is not an edge (in particular, $i_1=1$ and $i_2=2$).  
We claim that $x_{i_1}\dots x_{i_k}$ is a directed path on $k\ge \ceil{2n/3}$ vertices whose square is contained in the tournament.

\medskip
To see this, first observe that if the index $i+2$ is not in $I$, then both $i$ and $i+1$ are in $I$. Indeed, if $x_{i+2} x_i$ is an edge, then $x_{i+1}x_{i-1}$ cannot be one because of \eqref{prop:triple}, and $x_i x_{i-2}$ cannot be one because $i$ is not a bad index. This immediately implies $k\ge \ceil{2n/3}$.

It remains to check that $x_{i_{j-2}} x_{i_j}$ and $x_{i_{j-1}} x_{i_j}$ are all edges in the tournament. By the above observation, we know that $i_j -3 \le i_{j-2} < i_{j-1} < i_j$. Here $x_{i_j-1}x_{i_j}$ is an edge by \eqref{prop:edge}, and $x_{i_j-2}x_{i_j}$ is an edge by the definition of $I$. So the only case left is to show that $x_{i_{j-2}}x_{i_j}$ is an edge when $i_{j-2}=i_j-3$.

In this case there is an index $i_j-3 < i < i_j$ that is not in $I$, i.e., $x_ix_{i-2}$ is an edge in the tournament. But then if $i=i_j-1$, then $x_{i_{j-2}}x_{i_j}$ is an edge because of \eqref{prop:triple}, while otherwise $i=i_j-2$, and $x_{i_{j-2}}x_{i_j}$ is an edge because $i$ is not a bad index. This concludes our proof.
\end{proof}

\end{document}